\documentclass[11pt, reqno]{amsart}
\usepackage{color}
\usepackage{comment}
\usepackage[]{hyperref}
\usepackage{amsmath}
\usepackage{amsfonts}
\usepackage{amssymb}
\usepackage{amscd}
\usepackage{tikz}
	\usetikzlibrary{arrows,snakes}
	\usetikzlibrary{decorations.pathreplacing} 
	\usetikzlibrary{matrix}
\usepackage{stmaryrd}
\usepackage{multirow}
\usepackage{mathtools}
\usepackage{subfigure}
\usepackage{verbatim}
\usepackage{pinlabel}
\usepackage[all,cmtip,knot]{xy}
\usepackage{enumerate, enumitem, comment}
\usepackage[normalem]{ulem}

\newtheorem{theorem}{Theorem}[section]

\newtheorem{lemma}[theorem]{Lemma}
\newtheorem{proposition}[theorem]{Proposition}

\newtheorem{question}[theorem]{Question}
\newtheorem{conjecture}[theorem]{Conjecture}
\newtheorem*{namedtheorem}{\theoremname}
\newcommand{\theoremname}{testing}

\theoremstyle{definition}

\theoremstyle{remark}
\newtheorem{remark}[theorem]{Remark}
\newtheorem{example}[theorem]{Example}
\newtheorem*{ack}{Acknowledgements}

\def\R{\mathbb{R}}
\def\Z{\mathbb{Z}}
\def\Q{\mathbb{Q}}

\def\cL{\mathcal{L}}

\def\ref{\textup{ref}}

\def\rk{\textup{rk}}

\def\lk{\operatorname{lk}}

\def\rk{\operatorname{rk}}

\def\d{\partial}

\newcommand{\co}{\mskip0.5mu\colon\thinspace}

\def\varep{\varepsilon}

\author[Jennifer Hom]{Jennifer Hom}
\thanks{The author was partially supported by NSF grants DMS-1128155, DMS-1307879, and a Sloan Research Fellowship.}
\address {School of Mathematics, Georgia Institute of Technology, Atlanta, GA 30332}
\address{School of Mathematics, Institute for Advanced Study, Princeton, NJ 08540}
\email{hom@math.gatech.edu}

\numberwithin{equation}{section}

\title{Satellite knots and L-space surgeries}

\begin{document}
\maketitle

\begin{abstract}
We give sufficient conditions for a satellite knot to admit an L-space surgery, and use this result to give new infinite families of patterns which produce satellite L-space knots.
\end{abstract}

\section{Introduction}

An \emph{L-space} is a rational homology sphere with the simplest possible Heegaard Floer homology, i.e., $\rk \widehat{HF}(Y) = |H_1(Y; \Z)|$. An \emph{L-space knot} is a knot $K \subset S^3$ which admits a positive L-space surgery. Since lens spaces are L-spaces, any knot admitting a positive lens space surgery is an L-space knot. However, there are many L-space knots which do not admit lens space surgeries, e.g., \cite[Proposition 23]{BakerMoore}.

Satellite L-space knots were first studied by Hedden in \cite{HeddencablingII}. Let $K_{p,q}$ denote the $(p,q)$-cable of $K$, where $p$ denotes the longitudinal winding. Without loss of generality, we assume $p>1$. Hedden proved that if $K$ is an L-space knot and $q > p(2g(K)-1)$, then $K_{p,q}$ is an L-space knot; the author proved the converse in \cite{HomLspace}.

The author, Lidman, and Vafaee consider a broader class of patterns in \cite{HLV}, namely \emph{Berge-Gabai knots}, that is, knots $P \subset D^2 \times S^1$ which admit solid torus surgeries. They prove that a satellite knot with a Berge-Gabai pattern is an L-space knot if and only if the companion $K$ is an L-space knot and the Berge-Gabai knot is sufficiently twisted relative to the genus of $K$.

Our main result is a set of sufficient conditions for a satellite knot to be an L-space knot. Let $P(K, n)$ be the $n$-twisted satellite with pattern $P \subset D^2 \times S^1$ and companion $K$; that is, we obtain $P(K, n)$ by taking the union of $D^2 \times S^1$ and the complement of $K$, identifying $\d D^2 \times \{ \textup{pt} \}$ with a meridian of $K$ and $\{ \textup{pt} \} \times S^1$ with an $n$-framed longitude. (Throughout, we assume that $K$ is non-trivial and that $P$ is not contained in a $3$-ball in $D^2 \times S^1$.) At times, we will write $P(K)$ to denote $P(K, 0)$. Let $U$ denote the unknot and $w(P)$ the winding number of $P$. We say $K$ is a \emph{negative L-space knot} if $K$ admits a negative L-space surgery, or equivalently if $-K$ is an L-space knot.

\begin{theorem}\label{thm:sufficient}
The satellite knot $P(K)$ is an L-space knot if the following conditions are satisfied:
\begin{enumerate}
	\item the knot $K$ is an L-space knot,
	\item \label{it:disk} $w(P) \geq 2$ and there exists a meridional disk $D \subset D^2 \times S^1$ which intersects $P$ in exactly $w(P)$ points, 
	\item \label{it:2g} the knot $P(U, -2g)$ is an L-space knot, where $g$ is the genus of $K$,
	\item \label{it:negative} the knot $P(U, -n)$ is a negative L-space knot for all sufficiently large integers $n$.
\end{enumerate}
\end{theorem}

\begin{remark}
An affirmative answer to \cite[Question 1.16]{BakerMotegi} would imply that condition \eqref{it:disk} is implied by condition \eqref{it:negative}.
\end{remark}

\noindent Below are several infinite families of patterns $P$ satisfying the last three conditions of Theorem \ref{thm:sufficient}.

\begin{example}\label{ex:torusknots}
Sufficiently twisted torus knots, with their standard embedding into $D^2 \times S^1$, satisfy conditions \eqref{it:disk}--\eqref{it:negative}. Thus, Theorem \ref{thm:sufficient} implies that if $K$ is an L-space knot and $q \geq 2pg(K) - 1$, then $K_{p,q}$ is an L-space knot. Indeed, if $P$ is the $(p,q)$-torus knot in $D^2 \times S^1$ (where, as above, $p$ indicates the longitudinal winding), then $P(U, -2g(K))$ is the $(p, q-2pg(K))$-torus knot, which is an L-space knot exactly when $q-2pg(K) \geq -1$. By \cite[Theorem 1.10]{HeddencablingII} and \cite{HomLspace}, we have that $K_{p,q}$ is an L-space knot if and only if $K$ is an L-space knot and $q > p(2g(K)-1)$. In particular, this shows that Theorem \ref{thm:sufficient} does not give necessary conditions for a satellite to be an L-space knot.
\end{example}

\begin{example}\label{ex:1bridge}
A knot $P$ in $D^2 \times S^1$ is an \emph{$n$-bridge braid} if it can be isotoped to be a braid in $D^2 \times S^1$ which lies in $\d D^2 \times S^1$ except for $n$ bridges. By \cite{GLV}, we have that if $P$ is a 1-bridge braid in $D^2 \times S^1$, then $P(U)$ admits an L-space surgery. Furthermore, the property of being a 1-bridge braid is preserved under adding (positive or negative) full twists to $P$. It is straightforward to verify that a 1-bridge braid is either a positive or negative braid, e.g., by considering the classification of 1-bridge braids in \cite[Section 2]{Gabai1990}. If a 1-bridge braid $P$ is a positive braid, then $P(U)$ is a L-space knot; if $P$ is a negative braid, then $P(U)$ is a negative L-space knot. Thus, sufficiently positively twisted 1-bridge braids satisfy conditions \eqref{it:disk}--\eqref{it:negative} in Theorem \ref{thm:sufficient}.

Note that by \cite[Theorem 2.6]{Berge}, the set of Berge-Gabai knots forms a proper subset of the set of 1-bridge braids in $D^2 \times S^1$, and there are infinitely many 1-bridge braids with are not Berge-Gabai knots. In particular, sufficiently positive 1-bridge braid patterns give a new infinite family of satellite L-space knot patterns.
\end{example}

\begin{example}
For sufficiently large $n$, the knots $K_{n, 0} \subset S^3 \setminus N(c_a) \cong D^2 \times S^1$ and $K_{0, n} \subset S^3 \setminus N(c_b) \cong D^2 \times S^1$ in \cite[Theorem 8.1]{Motegi} satisfy conditions \eqref{it:disk}--\eqref{it:negative} in Theorem \ref{thm:sufficient}. (These patterns satisfy condition \eqref{it:disk} by the remark following \cite[Question 1.16]{BakerMotegi}.) The knots $K_{n, 0}, K_{0, n} \subset S^3$ are tunnel number two, and thus these patterns are distinct from Examples \ref{ex:torusknots} and \ref{ex:1bridge}.
\end{example}

\begin{example}
For sufficiently large $n$, the twisted torus knots in \cite[Theorem 1.8(1)]{Motegi} (viewed as a knot in the complement of the unknotted circle $c$ in \cite[Definition 5.2]{Motegi}) satisfy conditions \eqref{it:disk}--\eqref{it:negative} in Theorem \ref{thm:sufficient}.
\end{example}

In light of Example \ref{ex:torusknots}, we ask the following natural question:

\begin{question}
Can the conditions on $P$ in Theorem \ref{thm:sufficient} be relaxed to give necessary and sufficient conditions for $P(K)$ to be an L-space knot?
\end{question}

\noindent Regarding necessary conditions for $P(K)$ to be an L-space knot, note that by \cite{OSlens} and \cite{Ni}, L-space knots are fibered, and by \cite{HMS}, if $P(K)$ is fibered, then both $K$ and $P(U)$ are fibered and the winding number of $P$ is nonzero. Hence if $P(K)$ is an L-space knot, then both $K$ and $P(U)$ are fibered and $w(P) \neq 0$. See Section \ref{sec:fr} for further remarks related to necessary conditions for $P(K)$ to be an L-space knot.

We also consider the following question:

\begin{question}[{c.f. \cite[Question 22]{BakerMoore}}]
Do conditions \eqref{it:disk}--\eqref{it:negative} in Theorem \ref{thm:sufficient} imply that $P$ is a braid? A strongly quasipositive braid? A positive braid?
\end{question}

\begin{ack}
I would like to thank Ken Baker, Josh Greene, Tye Lidman, and Fery Vafaee for helpful conversations.
\end{ack}

\section{Proof of Theorem}

We prove Theorem \ref{thm:sufficient} using work of Hanselman, J. Rasmussen, S. Rasmussen, and Watson. Let $M$ be a compact, connected, oriented 3-manifold with torus boundary, and let $M(\alpha)$ denote the result of Dehn filling $M$ along a slope of $\alpha \subset \d M$, where $\alpha$ represents a primitive class in $H_1(\d M; \Z)/ \pm 1$. If we fix a basis for $H_1(\d M; \Z)$, we may identify the set of slopes with $\Q P^1 = \Q \cup \{ \frac{1}{0} \}$, viewed as a subspace of $\R P^1$. Let
\[ \cL(M) = \{ \alpha \mid M(\alpha) \textup{ is an L-space}\}. \]
We will be particularly interested in its interior, $\cL^\circ(M)$, the set of \emph{strict L-space slopes}.

\begin{theorem}[{\cite[Theorem 4]{HRRW}}]\label{thm:HRRW}
Let $M_1$ and $M_2$ be compact, connected, oriented 3-manifolds with torus boundary, and suppose that $Y \cong M_1 \cup_h M_2$ for some homeomorphism $h \co \d M_1 \rightarrow \d M_2$. 
If, for every slope $\alpha \subset \d M_1$, either $\alpha \in \cL^\circ(M_1)$ or $h(\alpha) \in \cL^\circ(M_2)$, then $Y$ is an L-space.
\end{theorem}

If $\cL(M)$ is nonempty, then $b_1(M ) = 1$, which implies that $M$ is a rational homology $D^2 \times S^1$. Let $[\lambda]$ denote the rational longitude of $M$, i.e., $\lambda$ is a primitive element of $H_1(\d M; \Z)$ such that $\iota(\lambda) \in H_1(M)$ is torsion.

\begin{theorem}[{\cite[Proposition 1.3 and Theorem 1.6]{RR}}]\label{thm:RR}
If $|\cL(M)| > 1$, then $\cL(M)$ is either $\Q P^1 \setminus \{ [\lambda] \}$ or a closed interval in $\Q P^1$.
\end{theorem}

We also recall the following proposition of Ozsv\'ath-Szab\'o:

\begin{proposition}[{\cite[Proposition 9.6]{OSrational}}]\label{prop:OS}
If $K$ is an L-space knot, then $S^3_{p/q}(K)$ is an L-space if and only if $p/q \geq 2g(K)-1$.
\end{proposition}

\noindent In other words, if $K$ is an L-space knot and $M_K = S^3 \setminus N(K)$, then $\cL^\circ(M_K) = (2g(K)-1, \infty)$, where we use the usual identification of slopes on the knot complement with $\Q \cup \{\frac{1}{0}\}$. If $K$ is a negative L-space knot, then $\cL^\circ(M_K) = (-\infty, -2g(K)+1)$. 

In what follows, we shift our perspective, and rather than view $P$ as a knot in $D^2 \times S^1$, we instead consider a link $P \cup J \subset S^3$ where $J$ is unknotted. To obtain a knot in the solid torus, we consider $P \subset S^3 \setminus N(J)$. See Figure \ref{fig:patternlink}. We write $\tilde{P}$ to denote the image of $P$ in surgery along $J$, and similarly $\tilde{J}$ for the image of $J$ in surgery along $P$. Let $S^3_{r,s}(P \cup J)$ denote the result of $r$-surgery along $P$ and $s$-surgery along $J$.

\begin{figure}[ht]
\labellist
\pinlabel {${P}$} at 140 104
\pinlabel {$J$} at 23 100
\endlabellist
\includegraphics[scale=0.8]{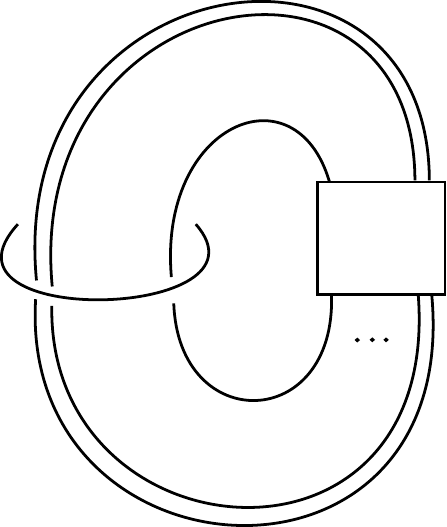}
\caption{The link $P \cup J$.}
\label{fig:patternlink}
\end{figure}

\begin{lemma}\label{lem:main}
Let $L = P \cup J$ be a link in $S^3$ such that
\begin{enumerate}
	\item $J$ is unknotted,
	\item \label{it:winding} $|\lk(P, J)| = w \geq 2$,
	\item \label{it:Jdisk} $J$ bounds a disk $D \subset S^3$ which intersects $P$ in exactly $w$ points.
	\newcounter{enumi_saved}
	\setcounter{enumi_saved}{\value{enumi}}
\end{enumerate}
If there exist positive integers $a$, $b$, and $r$ such that 
\begin{enumerate}
\setcounter{enumi}{\value{enumi_saved}}
	\item \label{it:ineqra} $r \geq 2g(P)+aw(2w-1)-1$, 
	\item  \label{it:ineqrb} $b \geq \frac{2g(P)+r-1}{w}$, 
	\item the knot $\tilde{P}_a \subset S^3_{1/a}(J) \cong S^3$ is an L-space knot,
	\item the knot $\tilde{P}_b \subset S^3_{1/b}(J) \cong S^3$ is a negative L-space knot,
\end{enumerate}
then $[-\infty, \frac{1}{b}] \cup [\frac{1}{a}, \infty] \subset \cL(M_{r})$, where $M_r=S^3_r(P) \setminus N(\tilde{J})$.
\end{lemma}

\begin{remark}\label{rem:Jslopes}
In the above lemma, we identify slopes on $M_r$, the complement of $\tilde{J}$, with $\Q \cup \{\frac{1}{0}\}$ such that filling $M_r$ along slope $s$ yields $S^3_{r,s}(P \cup J)$.
\end{remark}

\begin{proof}
We will show that if the hypotheses of the lemma are satisfied, then
\begin{itemize}
	\item $S^3_{r, w^2/r}(P \cup J)$ is not an L-space,
	\item $S^3_{r, 1/a}(P \cup J)$ and $S^3_{r, 1/b}(P \cup J)$ are L-spaces.
\end{itemize}
We now show that the two bullet points above imply the lemma. Note that
\begin{align*}
	r & \geq 2g(P)+aw(2w-1)-1 \\
		&\geq a(2w^2-w)-1 \\
		&> aw^2
\end{align*}
since $w \geq 2$ and $a$ is a positive integer. In particular, $r > 4$. We also have that
\begin{align*}
	b & \geq \frac{2g(P)+r-1}{w} \\
		&\geq \frac{r-1}{w} \\
		&> \frac{r}{w^2},
\end{align*}
where the last inequality follows from the fact that $r > 4$ and $w \geq 2$. Thus, we have shown that $0 < a < \frac{r}{w^2} < b$, i.e., $\frac{1}{b} < \frac{w^2}{r} < \frac{1}{a}$. It then follows from the bullet points above together with Theorem \ref{thm:RR} that $[-\infty, \frac{1}{b}] \cup [\frac{1}{a}, \infty] \subset \cL(M_{r})$, as desired.

Since $|\lk(P, J)| = w$, we have that $H_1(S^3_{r, w^2/r}(P \cup J); \Z) \cong \Z$, and hence $S^3_{r, w^2/r}(P \cup J)$ is not an L-space.

Next, we show that $S^3_{r, 1/a}(P \cup J)$ is an L-space. Note that
\[ S^3_{r, 1/a}(P \cup J) \cong S^3_{r-aw^2}(\tilde{P}_a). \]
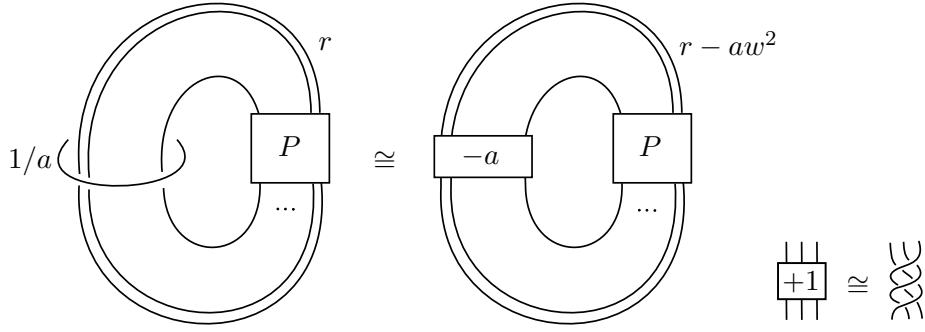
\begin{figure}[ht]
\labellist
\pinlabel {$1/a$} at -2 100
\pinlabel {$P$} at 120 107
\pinlabel {$r$} at 137 155
\pinlabel {$\cong$} at 165 100
\pinlabel {$-a$} at 210 102
\pinlabel {$P$} at 290 107
\pinlabel {$r-aw^2$} at 327 155
\pinlabel {$+1$} at 361 43
\pinlabel {$\cong$} at 388 41
\endlabellist
\includegraphics[scale=0.8]{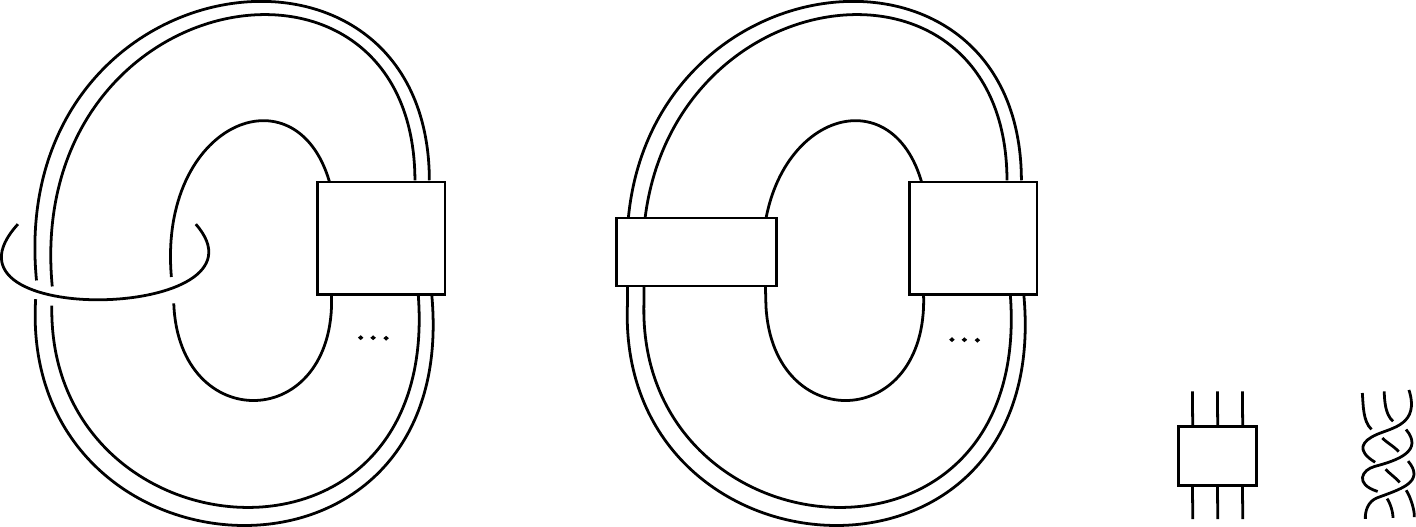}
\caption{Left, $S^3_{r, 1/a}(P \cup J)$. Right, $S^3_{r-aw^2}(\tilde{P}_a)$.}
\label{fig:patternlink2}
\end{figure}

\noindent See Figure \ref{fig:patternlink2}. Since $\tilde{P}_a \subset S^3_{1/a}(J) \cong S^3$ is an L-space knot, it follows that $S^3_{r, 1/a}(P \cup J)$ is an L-space if $r-aw^2 \geq 2g(\tilde{P}_a)-1$. We have that 
\[ g(\tilde{P}_a) \leq g(P) + \frac{aw(w-1)}{2}, \] 
since adding a (positive or negative) full twist to $P$ increases the genus by at most $\frac{w(w-1)}{2}$ by \eqref{it:winding} and \eqref{it:Jdisk}, and $\tilde{P}_a$ is obtained from $P$ by $a$ negative full twists. Hence,
\begin{align*}
	2g(\tilde{P}_a) -1 &\leq 2g(P) + aw(w-1) - 1 \\ 
	&\leq r-aw^2,
\end{align*}
where the second inequality follows from \eqref{it:ineqra}. Hence $S^3_{r, 1/a}(P \cup J)$ is an L-space.

Similarly, we show that $S^3_{r, 1/b}(P \cup J)$ is an L-space. Note that
\[ S^3_{r, 1/b}(P \cup J) \cong S^3_{r-bw^2}(\tilde{P}_b). \]
Since $\tilde{P}_b \subset S^3_{1/b}(J) \cong S^3$ is a \emph{negative} L-space knot, it follows that $S^3_{r, 1/b}(P \cup J)$ is an L-space if $r-bw^2 \leq -2g(\tilde{P}_b)+1$. We have that 
\[ g(\tilde{P}_b) \leq g(P) + \frac{bw(w-1)}{2} \] 
and so
\begin{align*}
-2g(\tilde{P}_b) +1 &\geq -2g(P) -bw(w-1) +1 \\
	&\geq r-bw^2, \\
\end{align*}
where the second inequality follows from \eqref{it:ineqrb}. Hence $S^3_{r, 1/b}(P \cup J)$ is an L-space, completing the proof of the lemma.
\end{proof}

\begin{proof}[Proof of Theorem \ref{thm:sufficient}]
Let $M_K = S^3 \setminus N(K)$. If $K$ is an L-space knot, then by Proposition \ref{prop:OS}, we have that $\cL^\circ(M_K) = (2g(K)-1, \infty)$.

Let $P \subset D^2 \times S^1$ be a pattern satisfying the hypotheses of the theorem, and let $L = P \cup J$ be the associated link. Let $a=2g(K)$, and choose $r$ and $b$ sufficiently large so that they satisfy \eqref{it:ineqra} and \eqref{it:ineqrb} in Lemma \ref{lem:main}. Then Lemma \ref{lem:main} implies that 
\begin{equation}
\label{eqn:intervals}
[-\infty, \frac{1}{b}] \cup [\frac{1}{2g(K)}, \infty] \subset \cL(M_r).
\end{equation}

We consider the closed manifold obtained by gluing $M_r$ and $M_K$ via $h \co \d M_r \rightarrow \d M_K$ which identifies the meridian of $J$ with the 0-framed longitude of $K$ and the 0-framed longitude of $J$ with the meridian of $K$. Note that $M_r \cup_h M_K \cong S^3_r(P(K))$.

It remains to compare $\cL^\circ(M_K)$ and $h(\cL^\circ(M_r))$. Note that given a curve of slope $\frac{p}{q}$ on $\d M_r$, the corresponding slope on $\d M_K$ is $h(\frac{p}{q}) = \frac{q}{p}$. Namely, 
\begin{equation*}
	h\Big([-\infty, \frac{1}{b}) \cup (\frac{1}{2g(K)}, \infty]\Big) = [-\infty, 2g(K)) \cup (b, \infty].
\end{equation*}
Since $\cL^\circ(M_K) \cup h(\cL^\circ(M_r)) = \Q \cup \{\frac{1}{0}\}$, Theorem \ref{thm:HRRW} implies that $r$-surgery along $P(K)$ is an L-space, i.e., $P(K)$ is an L-space knot, as desired.
\end{proof}

\begin{remark}\label{rem:n-2g}
It follows from \eqref{eqn:intervals} that for a pattern $P$ satisfying conditions \eqref{it:disk}--\eqref{it:negative} in Theorem \ref{thm:sufficient}, the knot $P(U, n)$ is an L-space knot for all $n \geq -2g(K)$. Indeed, the knot $P(U, n)$ is isotopic to $\tilde{P}_{-n} \subset S^3_{1/{-n}}(J) \cong S^3$, and $S^3_{r, 1/-n}(P \cup J) \cong S^3_{r+nw^2}(P(U,n))$, which is an L-space whenever $\frac{1}{-n} \in \cL(M_r)$. Thus, $P(U,n)$ is an L-space knot if $\frac{1}{-n} \in \cL(M_r)$ and $r+nw^2$ is positive; these conditions are satisfied when $n \geq -2g(K)$.
\end{remark}

\section{Further Remarks} \label{sec:fr}

Note that in the proof of Theorem \ref{thm:sufficient}, we found that $\cL^\circ(M_K) = (2g(K)-1, \infty)$ and $h(\cL^\circ(M_r)) \supset [-\infty, 2g(K)) \cup (b, \infty]$. The fact that the conditions in Theorem \ref{thm:sufficient} are not necessary conditions for $P(K)$ to be an L-space knot is likely due to the overlap between $\cL^\circ(M_K)$ and $h(\cL^\circ(M_r)$ in the interval $(2g(K)-1, 2g(K))$.

Regarding necessary conditions for $P(K)$ to be an L-space knot, the author, Lidman, and Vafaee made the following conjecture (c.f. \cite[Question 22]{BakerMoore}):

\begin{conjecture}[{\cite[Conjecture 1.7]{HLV}}] \label{con:HLV}
If $P(K)$ is an L-space knot, then so are $P(U)$ and $K$.
\end{conjecture}

We recall Conjecture 1.7 of J. Rasmussen and S. Rasmussen (c.f. \cite{BoyerClay}, \cite{Hanselman}):

\begin{conjecture}[\cite{RR}] \label{con:RR}
Let $M_1, M_2$ be compact, connected, oriented 3-manifolds with torus boundary. If $M_1, M_2$ are boundary incompressible, then $M_1 \cup_h M_2$ is an L-space if and only if $\cL^\circ(M_2) \cup h(\cL^\circ(M_1)) = \Q P^1$.
\end{conjecture}

\begin{proposition}\label{prop:RR}
Suppose Conjecture \ref{con:RR} is true. Let $P(K)$ be an L-space knot. Then $P(U)$ and $K$ are L-space knots. Furthermore, $P(U, n)$ is an L-space knot for all $n \geq -2g(K)+1$, and $P(U, -N)$ is a negative L-space knot for all sufficiently large $N$.
\end{proposition}

\begin{proof}
Throughout this proof, we assume that Conjecture \ref{con:RR} is true and that $P(K)$ is an L-space knot.

As above, rather than $P \subset D^2 \times S^1$, we consider the link $P \cup J \subset S^3$ where $J$ is unknotted. Then $P(K)$ is the image of $P$ in $S^3 \setminus N(J) \cup_h S^3 \setminus N(K)$, where $h$ identifies the longitude (respectively meridian) of $J$ with the meridian (respectively longitude) of $K$. As above, let $M_K = S^3 \setminus N(K)$ and $M_r = S^3_r(P) \setminus N(\tilde{J})$. We identify slopes on $\d M_r$ with $\Q \cup \{\frac{1}{0}\}$ so that filling $M_r$ along slope $s$ yields $S^3_{r,s}(P \cup J)$, as in Remark \ref{rem:Jslopes}. 

Since $K$ is assumed to be nontrivial, $M_K$ is boundary incompressible. Scharlemann \cite{Scharlemann} (extending work of Gabai \cite{Gabai1990}; see \cite[Section 4.1]{BakerMotegi} for a summary of the results) shows that $M_r$ is either 
\begin{enumerate}
	\item \label{it:solidtorus} a solid torus (and so $P$ is a 0- or 1-bridge braid in the solid torus),
	\item \label{it:Lpq} homeomorphic to $W \# L(p, q)$, and $P$ is a $(p,q)$-cable knot and $r$ is the cabling slope,
	\item irreducible and boundary irreducible.
\end{enumerate}
In case \eqref{it:solidtorus}, the resulting satellite knots are either cables or Berge-Gabai satellite knots, and the proposition follows from \cite{HeddencablingII, HomLspace} and \cite{HLV} respectively. In what follows, we will consider $M_r$ for sufficiently large $r$, which, if we are in case \eqref{it:Lpq}, can be taken to be larger than the cabling slope. In particular, we may assume that $M_r$ is boundary incompressible.

Since $S^3_0(K)$ is not an L-space for any $K$, we have that $0 \notin \cL^\circ(M_K)$. By Conjecture \ref{con:RR}, it follows that for sufficiently large $r$, the slope $h^{-1}(0) = \infty \in \cL^\circ(M_r)$. Hence $S^3_{r, \infty}(P \cup J) = S^3_r(P(U))$ is an L-space and so $P(U)$ is an L-space knot.

If $K$ is neither an L-space knot nor a negative L-space knot, then $\cL^\circ(M_K)$ is empty. By Theorem \ref{thm:RR}, the set $\cL^\circ(M)$ can never be all of $\Q P^1$, so if $P(K)$ is an L-space knot, then either $K$ or $-K$ must be an L-space knot. Since L-space knots are strongly quasipositive \cite{Hedden2010} and fibered, and the fiber surface for $P(K)$ is obtained by taking $w(P)$ parallel copies of the fiber surface for $K$ together with the fiber surface for $P$ in the solid torus, it follows that $K$ (and not $-K$) must be an L-space knot.

For an L-space knot $K$, we have that $\cL^\circ(M_K) = (2g(K)-1, \infty)$. Fix $r$ sufficiently large. By Conjecture \ref{con:RR}, it follows that
\[ h^{-1}([-\infty, 2g(K)-1]) = [-\infty, 0] \cup [\frac{1}{2g(K)-1}, \infty] \] 
is contained in $\cL^\circ(M_r)$, i.e., there exists $\varep > 0$ such that $S^3_{r, s}(P \cup J)$ is an L-space for all $s \in [-\infty, \varep] \cup [\frac{1}{2g(K)-1}-\varep, \infty]$.  We proceed as in Lemma \ref{lem:main}, and consider $\tilde{P}_a \subset S^3_{1/a}(J)$, the image of $P$ in $\frac{1}{a}$ surgery along $J$, for $a \in \Z$. Note that $S^3_{1/a}(J) \cong S^3$ and $\tilde{P}_a$ is isotopic to $P(U, -a)$. We have that
\[ S^3_{r, 1/a}(P \cup J) \cong S^3_{r-aw^2}(\tilde{P}_a), \]
where $w=w(P)$. For $\frac{1}{a} \in [-\infty, \varep] \cup [\frac{1}{2g(K)-1}-\varep, \infty]$, the manifold $S^3_{r-aw^2}(\tilde{P}_a)$ must be an L-space. Equivalently, for $a \leq 2g(K)-1$ or $a \geq \frac{1}{\varep}$, the manifold $S^3_{r-aw^2}(P(U, -a))$ is an L-space. In the former case, for $r$ sufficiently large, $r-aw^2$ is positive, and so $P(U, -a)$ is an L-space knot. In the latter case, if $a \geq \frac{r}{w^2}$, then $r-aw^2$ is negative, and so $P(U, -a)$ is a negative L-space knot.
\end{proof}

\bibliographystyle{amsalpha}
\bibliography{references}

\end{document}